\documentclass[10pt,a4paper]{article}
\usepackage[a4paper,margin=2cm]{geometry}
\usepackage[singlespacing]{setspace} 
\usepackage[hyphens]{url}
\usepackage[bookmarks=true,bookmarksopen=true]{hyperref}
\usepackage{enumitem}
\usepackage{amsmath,amsfonts,amssymb,amsthm,amsbsy}
\usepackage{color}
\usepackage{graphicx}
\usepackage[ruled,linesnumbered]{algorithm2e}
\usepackage{multirow}
\usepackage{subcaption}
\usepackage[sort&compress,square,numbers]{natbib}
\usepackage{tcolorbox}

\theoremstyle{plain}

\newcounter{LEM}
\newtheorem{lem}[LEM]{Lemma}

\theoremstyle{definition}
\newcounter{DEFN}
\newtheorem{defn}[DEFN]{Definition}

\theoremstyle{remark}
\newcounter{REM}
\newtheorem{rem}[REM]{Remark}

\tcbset{colback=white!96!black,colframe=teal!30!white, width=\hsize, boxrule=0.3mm, left=0mm}

\begin{document}
	\title{\textbf{A Passivity-Based Method for \\ Accelerated Convex Optimisation}\footnote{\copyright 2024 IEEE. Personal use of this material is permitted. Permission from IEEE must be obtained for all other uses, in any current or future media, including reprinting/republishing this material for advertising or promotional purposes, creating new collective works, for resale or redistribution to servers or lists, or reuse of any copyrighted component of this work in other works}} 
	\author{Namhoon Cho \footnote{Research Fellow, Centre for Autonomous and Cyber-Physical Systems, School of Aerospace, Transport and Manufacturing, Cranfield University, Cranfield, Bedfordshire, MK43 0AL, United Kingdom. \texttt{n.cho@cranfield.ac.uk}}, 
		and Hyo-Sang Shin \footnote{Professor, Cho Chun Shik Graduate School of Mobility, Korea Advanced Institute of Science and Technology, Daejeon, 34051, South Korea. \texttt{hyosangshin@kaist.ac.kr}}
	}
	\date{}
	\maketitle
	
	\begin{abstract}
		This study presents a constructive methodology for designing accelerated convex optimisation algorithms in continuous-time domain. The two key enablers are the classical concept of passivity in control theory and the time-dependent change of variables that maps the output of the internal dynamic system to the optimisation variables. The Lyapunov function associated with the optimisation dynamics is obtained as a natural consequence of specifying the internal dynamics that drives the state evolution as a passive linear time-invariant system. The passivity-based methodology provides a general framework that has the flexibility to generate convex optimisation algorithms with the guarantee of different convergence rate bounds on the objective function value. The same principle applies to the design of online parameter update algorithms for adaptive control by re-defining the output of internal dynamics to allow for the feedback interconnection with tracking error dynamics.
	\end{abstract}
	
	
	\section{Introduction} \label{Sec:Intro}
	The analysis of accelerated optimisation algorithms in the continuous-time domain has been the area of active research in recent years. Well-known results include the analysis of Nesterov's accelerated gradient descent method provided in \cite{Su2016} by taking the limit of infinitesimal step-size. In relation to continuous-time optimisation as well as direct adaptive control, \cite{Gaudio2021} presented four different algorithms and studied their stability properties. The recent work of \cite{Suh2022} provided a much more generalised understanding towards various accelerated convex optimisation methods by highlighting the role of using a time-dilated coordinate in the analysis. However, the scope of \cite{Su2016,Gaudio2021,Suh2022} is focused on the second-order optimisation dynamics of specific forms.
	
	Meanwhile, there have been attempts to approach optimisation from robust control perspectives \cite{Hu2017a, Hu2017b, Lessard2022}. The control-theoretic concept of dissipativity was noticed to be central in the construction of associated Lyapunov function, which in turn enables rigorous convergence analysis. Although the role of dissipativity theory in the analysis and design of optimisation algorithms was illustrated very clearly, the scope of \cite{Hu2017a,Hu2017b,Lessard2022} alone falls short of providing a constructive procedure for achieving accelerated convergence in the design of a new optimisation algorithm.
	
	Motivated by the viewpoints provided by the previous works, we propose a constructive methodology for streamlined design and analysis of accelerated convex optimisation algorithms. The proposed approach combines the utility of the two design philosophies, namely, the coordinate transform of \cite{Suh2022} and the dissipativity viewpoint of \cite{Lessard2022}. Central to the new development is the idea that the passive input-output pair of a open-loop stable linear time-invariant dynamic system can be defined to be the gradient of an objective function and the vector of optimisation variables, respectively. Our new method characterises the structure of passive optimisation dynamics to a generalised extent while guaranteeing a prescribed order for the convergence rate bound on the objective function value.
	
	\section{Preliminaries} \label{Sec:Prelim}
	This section introduces several basic concepts as the mathematical preliminaries and the notation. 
	
	\subsection{Convex Optimisation}
	Consider a differentiable and convex function $f: \mathbb{R}^{n} \rightarrow \mathbb{R}$. The optimisation problem of our interest is described as
	\begin{equation} \label{Eq:prob}
		\underset{\theta\in\mathbb{R}^{n}}{\mathrm{minimise}} \quad f\left(\theta\right)
	\end{equation}
	If the solution of the problem in Eq. \eqref{Eq:prob} exists, let $\theta_{*} \in \mathbb{R}^{n}$ represents the optimal solution. Also, let $f_{*} \triangleq \inf\limits_{\theta\in\mathbb{R}^{n}} f\left(\theta\right)$ be the optimal value of the objective function.
	
	The Bregman divergence associated with $f$ for points $p, q \in \mathbb{R}^{n}$ is defined as
	\begin{equation} \label{Eq:Bregman_defn}
		D_{f}\left(p,q\right) \triangleq f\left(p\right)-f\left(q\right) - \left<\nabla f\left(q\right),p-q\right>
	\end{equation}
	where $\left< \cdot,\cdot \right>$ denotes the inner product. By definition of convexity, for a differentiable and convex function $f$, we have
	\begin{equation} \label{Eq:Bregman_Ieq}
		D_{f}\left(p,q\right) \geq 0
	\end{equation}
	for all $p, q \in \mathbb{R}^{n}$.
	
	\subsection{Stability and Passivity}
	Consider a linear time-invariant dynamic system whose state space realisation is given by
	\begin{equation} \label{Eq:sys_lti}
		\mathcal{G}:\qquad 
		\begin{aligned}
			\dot{x}\left(t\right) &= Ax\left(t\right) + Bu\left(t\right)\\
			y\left(t\right)&=Cx\left(t\right)+Du\left(t\right)
		\end{aligned}
	\end{equation}
	with constant matrices $A$, $B$, $C$, and $D$, where $x$ represents the state, $u$ represents the input, and $y$ represents the output. One may consider $\mathcal{G}\left[\cdot\right]$ as an operator that maps $u$ to $y$ as $y = \mathcal{G}\left[u\right]$.
	
	The following lemma is a classical result of linear system theory which relates the stability of a system that has no input-output with a Lyapunov equation.
	
	\begin{lem}[Lyapunov Equation, Theorem 3.7 in \cite{Khalil2015}] \label{Lem:Lyap}
		A matrix $A$ is Hurwitz if and only if for every positive definite symmetric matrix $Q$ there exists a positive definite symmetric matrix $P$ that satisfies the Lyapunov equation
		\begin{equation} \label{Eq:Lyap}
			A^{T}P + PA = -Q
		\end{equation}
		Moreover, if $A$ is Hurwitz, then $P$ is the unique solution.
	\end{lem}
	
	Passivity is a fundamental concept in classical control theory which formally describes the property of a system that does not generate energy internally but only stores and dissipates energy supplied from outside \cite{Kottenstette2014}. 
	
	\begin{defn}[Passivity, Definition 5.3 in \cite{Khalil2015}] \label{Defn:passivity}
		Consider a dynamic system described by
		\begin{equation} \label{Eq:sys_nl}
			\begin{aligned}
				\dot{x}\left(t\right) &= f\left(x\left(t\right), u\left(t\right)\right)\\
				y\left(t\right) &= h\left(x\left(t\right), u\left(t\right)\right)
			\end{aligned}
		\end{equation}
		The system in Eq. \eqref{Eq:sys_nl} is passive if there exists a continuously differentiable positive semidefinite storage function $V\left(x\right)$ such that
		\begin{equation} \label{Eq:passivity_IEq}
			u^{T}y \geq \dot{V} = \left<\nabla_{x}V, f\left(x,u\right)\right>\qquad \forall \left(x,u\right)
		\end{equation}
		Moreover, it is strictly passive if $u^{T}y \geq \dot{V} + \phi\left(x\right)$ for some positive definite function $\phi$.
	\end{defn}
	
	The following lemma is another classical result of linear system theory which relates the passivity of a system that has input-output with a set of algebraic equations extending Eq. \eqref{Eq:Lyap}.
	
	\begin{lem}[Kalman-Yakubovich-Popov, Lemma 5.3 and 5.4 in \cite{Khalil2015}] \label{Lem:KYP} 
		Let $G\left(s\right) = C\left(sI-A\right)^{-1}B+D$ be a square transfer function matrix, where $\left(A,B\right)$ is controllable and $\left(A,C\right)$ is observable. Then, $G\left(s\right)$ is (strictly) positive real if and only if there exist matrices $P=P^{T}>0$, $L$, and $W$, (and a positive constant $\epsilon$) such that
		\begin{align}
			PA+A^{T}P&=-L^{T}L \left( -\epsilon P \right) \label{Eq:KYP_A}\\
			PB&=C^{T}-L^{T}W \label{Eq:KYP_BC}\\
			W^{T}W&=D+D^{T} \label{Eq:KYP_D}
		\end{align}
		Moreover, the system with minimal realisation $\left(A,B,C,D\right)$ is (strictly) passive if $G\left(s\right)$ is (strictly) positive real. The storage function certifying passivity of the system is given by 
		\begin{equation} \label{Eq:V_storage}
			V = \frac{1}{2}x^{T}Px
		\end{equation}
	\end{lem}
	
	\begin{proof}
		For $V$ given by Eq. \eqref{Eq:V_storage}, straightforward substitution of Eqs. \eqref{Eq:KYP_A}-\eqref{Eq:KYP_D} verifies
		\begin{equation} \label{Eq:uy-V_dot}
			\begin{aligned}
				u^{T}y - \dot{V} &= u^{T}\left(Cx+Du\right) - x^{T}P\left(Ax+Bu\right)\\
				&= \frac{1}{2}\left(Lx+Wu\right)^{T}\left(Lx+Wu\right)+\frac{1}{2}\epsilon x^{T}Px \\
				&\geq \epsilon V
			\end{aligned}
		\end{equation}
	\end{proof}
	
	Readers may refer to \cite{Madeira2016} for more detailed discussion on the equivalence between the frequency-domain notion of strict positive realness and the time-domain notion of strict passivity for linear dynamic systems.
	
	\subsection{Partial Derivative of Time-Dependent Function}
	Motivated by \cite{Suh2022}, consider a scalar-valued function $U\left(v\left(t\right),t\right)$ where vector-valued $v\left(t\right)$ is differentiable with respect to $t$. The chain rule states that
	\begin{equation} \label{Eq:chain_rule_diff}
		\frac{d}{dt}U\left(v\left(t\right),t\right) = \left<\nabla_{v}U\left(v,t\right), \dot{v}\left(t\right)\right> + \frac{\partial}{\partial t}U\left(v,t\right)
	\end{equation}
	Integrating Eq. \eqref{Eq:chain_rule_diff} from $t_{0}$ to $t$ gives
	\begin{equation} \label{Eq:chain_rule_int}
			\int_{t_{0}}^{t}\left<\nabla_{v}U\left(v,\tau\right), \dot{v}\left(\tau\right)\right>d\tau = U\left(v\left(t\right),t\right) - U\left(v\left(t_{0}\right),t_{0}\right) - \int_{t_{0}}^{t}\frac{\partial}{\partial \tau}U\left(v,\tau\right)d\tau
	\end{equation}

	\section{Passivity-Based Design of Optimisers} \label{Sec:KYP_Opt_design}
	\subsection{Generic Method}
	Suppose that $U\left(v\left(t\right), t\right): \mathbb{R}^{n}\times \mathbb{R} \rightarrow \mathbb{R}$ is a lower-bounded function (which will be specified later) whose value is nonnegative for every $t$ and defined with respect to some variable $v\left(t\right)\in \mathbb{R}^{n}$ (which will also be specified later). Let us introduce a variable $w\left(t\right) \in \mathbb{R}^{n}$ where $n= \dim\left(\theta\right)$ which evolves over time according to
	\begin{equation} \label{Eq:w_sys}
			w^{\left(m\right)}\left(t\right) + a_{m-1}w^{\left(m-1\right)}\left(t\right) +\cdots  +  a_{1}\dot{w}\left(t\right)+a_{0}w\left(t\right)= -\alpha\left(t\right)\nabla_{v} U\left(v\left(t\right),t\right)
	\end{equation}
	where $\left\{a_{i}\right\}_{i=0}^{n-1}$ is a set of constant coefficients, $\alpha\left(t\right)\neq 0$ for $\forall t\geq t_{0}$, and $w^{\left(i\right)}$ denotes the $i$-th time-derivative of $w$. 
	
	Let us refer to the system in Eq. \eqref{Eq:w_sys} as the generator dynamics $\mathcal{G}$ for which state and input are defined as
	\begin{equation} \label{Eq:xu}
		\begin{aligned}
			x &= \begin{bmatrix}
				w\\
				\dot{w} \\
				\ddot{w} \\
				\vdots \\
				w^{\left(m-1\right)}
			\end{bmatrix} \in \mathbb{R}^{mn}\\
			u &= -\alpha\nabla_{v} U\left(v,t\right) \in \mathbb{R}^{n}
		\end{aligned}
	\end{equation}
	Then, the ordinary differential equation in Eq. \eqref{Eq:w_sys} can be rewritten as a linear system of the form in Eq. \eqref{Eq:sys_lti} with the controllable canonical form realisation given by
	\begin{equation} \label{Eq:CCF}
				A = \begin{bmatrix}
					0 & 1 & 0 & \cdots & 0\\
					0 & 0 & 1 & \cdots & 0\\
					\vdots & \vdots & \vdots & \ddots & \vdots\\
					0 & 0 & 0 & \cdots & 1\\
					-a_{0} & -a_{1} & -a_{2} & \cdots & -a_{m-1}\\
				\end{bmatrix}
				\otimes I_{n}, \qquad
				B = \begin{bmatrix}
					0\\
					0\\
					\vdots\\
					0\\
					1
				\end{bmatrix}\otimes I_{n}
	\end{equation}
	where $\otimes$ represents the Kronecker product. 
	
	Suppose that $A$ is Hurwitz by appropriate choice of $\left\{a_{i}\right\}_{i=0}^{n-1}$ and $Q$ is an arbitrarily chosen symmetric positive definite matrix. Given strictly stable $A$ and $B$ as defined in Eq. \eqref{Eq:CCF}, Lemma \ref{Lem:Lyap} states that there exists symmetric positive definite $P$ satisfying Eq. \eqref{Eq:Lyap}. As noted in \cite{Slotine1991}, given the definition of input $u$, Lemma \ref{Lem:KYP} implies that the arbitrariness of $Q$ carries over to the possibility of constructing an infinite number of passive systems by compatible choices of outputs. Let the output has no direct feedthrough term, i.e., $D=0$. Then, Lemma \ref{Lem:KYP} verifies that the compatible output can be obtained by setting
	\begin{equation} \label{Eq:C_KYP}
		C=B^{T}P
	\end{equation}
	In this setup, the time-derivative of the storage function $V=\frac{1}{2}x^{T}Px$ can be written as
	\begin{equation} \label{Eq:V_dot}
		\begin{aligned}
			\dot{V} &= x^{T}P\left(Ax+Bu\right) = x^{T}PBu -\frac{1}{2}x^{T}Qx\\
			&= \left<y,u\right> -\frac{1}{2}x^{T}Qx
		\end{aligned}	
	\end{equation}
	Equivalently, if one defines 
	\begin{equation} \label{Eq:W_defn}
		W\left(t\right) \triangleq V\left(t\right) - \int_{t_{0}}^{t}\left<y\left(\tau\right),u\left(\tau\right)\right>d\tau
	\end{equation}
	then Eq. \eqref{Eq:V_dot} verifies that 
	\begin{equation} \label{Eq:W_dot}
		\begin{aligned}
		\dot{W} &= \dot{V} - \left<y,u\right> = -\frac{1}{2}x^{T}Qx \\ 
		&\leq 0
		\end{aligned}
	\end{equation}
	as per the definition of passivity stated in Definition \ref{Defn:passivity}. (We will come back to Eq. \eqref{Eq:W_dot} after defining the input-output map.)
	
	Integrating Eq. \eqref{Eq:V_dot} gives
	\begin{equation} \label{Eq:V_t_0}
			V\left(t\right) - V\left(t_{0}\right) 
			= \int_{t_{0}}^{t}\left<y\left(\tau\right),u\left(\tau\right)\right>d\tau - \frac{1}{2}\int_{t_{0}}^{t}x^{T}\left(\tau\right)Qx\left(\tau\right)d\tau
	\end{equation}
	Considering the way that the input $u$ is defined in Eq. \eqref{Eq:xu} and the integral relation shown in Eq. \eqref{Eq:chain_rule_int}, a natural choice is to define the variable $v\left(t\right)$ as the integral of scaled output, that is, 
	\begin{equation} \label{Eq:v_dot}
		\dot{v}\left(t\right) = \alpha\left(t\right)y\left(t\right) = \alpha\left(t\right)Cx\left(t\right)
	\end{equation}
	With this choice, Eq. \eqref{Eq:V_t_0} can be rewritten by using Eq. \eqref{Eq:chain_rule_int} as
	\begin{equation} \label{Eq:V_t_0_y}
			V\left(t\right) - V\left(t_{0}\right) = -U\left(v\left(t\right), t\right) + U\left(v\left(t_{0}\right), t_{0}\right)
			+ \int_{t_{0}}^{t} \frac{\partial}{\partial \tau}U\left(v,\tau\right)d\tau - \frac{1}{2}\int_{t_{0}}^{t}x^{T}\left(\tau\right)Qx\left(\tau\right)d\tau
	\end{equation}
	Note that if an invertible preconditioner matrix $M$ is introduced in the definition of input as $u=-\alpha M \nabla_{v}U$, then the inner product between input and output does not change by defining $\dot{v}=\alpha M^{-1} y$.
	
	We are now at the stage to relate i) the system variable $v\left(t\right)$ to the optimisation variable $\theta$ and ii) the auxiliary function $U$ to the objective function $f$. The purpose is to associate the construction of the passive dynamic system with the design of an optimisation algorithm to solve the problem in Eq. \eqref{Eq:prob}. In a similar manner to \cite{Suh2022}, let us introduce a coordinate transform given by
	\begin{equation} \label{Eq:theta_v}
		v\left(t\right)=\gamma\left(t\right)\left(\theta\left(t\right) - \theta_{c}\right)
	\end{equation}
	with a strictly monotonically increasing positive-valued function $\gamma\left(t\right) \in \mathbb{R}_{>0}$, i.e., $\dot{\gamma}\left(t\right)>0$ for $\forall t \geq t_{0}$, and some constant $\theta_{c}\in\mathbb{R}^{n}$. Since $\theta\left(v\left(t\right),t\right) = \gamma^{-1}\left(t\right)v\left(t\right) + \theta_{c}$, we have
	\begin{equation} \label{Eq:part_theta_t}
		\begin{aligned}
			\frac{\partial \theta\left(v,t\right)}{\partial t} &= -\gamma^{-2}\left(t\right)\dot{\gamma}\left(t\right)v\left(t\right) \\
			&= -\gamma^{-1}\left(t\right)\dot{\gamma}\left(t\right)\left(\theta\left(v,t\right)-\theta_{c}\right)
		\end{aligned}
	\end{equation}
	In anticipation of the Bregman divergence $D_{f}\left(\cdot,\cdot\right)$ appearing in the subsequent derivation, let us also define the conjugate function $U$ as 
	\begin{equation} \label{Eq:U_f}
			U\left(v\left(t\right),t\right) 
			= \gamma\left(t\right) \left[f\left(\theta\left(v\left(t\right),t\right)\right) - f\left(\theta_{c}\right)\right] + \psi\left(v\left(t\right),t\right)
	\end{equation}
	with a function $\psi\left(v\left(t\right),t\right)$ which satisfies $\psi\left(v\left(t\right),t\right)\geq \underline{\psi}$ for some constant $\underline{\psi}$ and $\frac{\partial}{\partial t}\psi\left(v\left(t\right),t\right) \leq 0$ for $\forall t\geq t_{0}$, and the choice of $\theta_{c}$ compatible with the nonnegativity of $f\left(\theta\left(v\left(t\right),t\right)\right) - f\left(\theta_{c}\right)$ such as $\theta_{c}=\theta_{*}$. Then, using Eq. \eqref{Eq:part_theta_t} in the partial differentiation of Eq. \eqref{Eq:U_f} with respect to $t$ yields
	\begin{equation} \label{Eq:part_U_t}
		\begin{aligned}
			\frac{\partial}{\partial t}U\left(v\left(t\right),t\right) &= \dot{\gamma}\left(t\right)\left[f\left(\theta\left(v\left(t\right),t\right)\right) - f\left(\theta_{c}\right)\right] +\gamma\left(t\right)\left<\nabla_{\theta}f\left(\theta\right),\frac{\partial \theta\left(v,t\right)}{\partial t}\right> + \frac{\partial}{\partial t}\psi\left(v\left(t\right),t\right)\\
			&=-\dot{\gamma}\left(t\right) D_{f}\left(\theta_{c},\theta\left(v,t\right)\right) + \frac{\partial}{\partial t}\psi\left(v\left(t\right),t\right) \\
			&\leq 0
		\end{aligned}	
	\end{equation}
	Note that the gradient of $U$ in Eq. \eqref{Eq:U_f} with respect to $v$ can be written as
	\begin{equation} \label{Eq:grad_U}
		\begin{aligned}
			\nabla_{v}U\left(v\left(t\right),t\right) &= \gamma\left(t\right)\nabla_{v}f\left(\gamma^{-1}\left(t\right)v\left(t\right)+\theta_{c}\right) + \nabla_{v}\psi\left(v,t\right)\\
			&= \nabla_{\theta}f\left(\theta\left(v\left(t\right),t\right)\right) + \nabla_{v}\psi\left(v,t\right)
		\end{aligned}
	\end{equation}
	
	Substituting Eq. \eqref{Eq:part_U_t} back into Eq. \eqref{Eq:V_t_0_y} and rearranging the equation gives a conservation relation as
	\begin{equation} \label{Eq:consv}
		\begin{aligned}
			E &\equiv U\left(v\left(t_{0}\right),t_{0}\right) + V\left(t_{0}\right)\\
			&= U\left(v\left(t\right),t\right) + V\left(t\right)  + \int_{t_{0}}^{t}\dot{\gamma}\left(\tau\right)D_{f}\left(\theta_{c},\theta\left(v,\tau\right)\right)d\tau  -\int_{t_{0}}^{t} \frac{\partial}{\partial \tau}\psi\left(v,\tau\right)d\tau + \frac{1}{2}\int_{t_{0}}^{t}x^{T}\left(\tau\right)Qx\left(\tau\right)d\tau\\
			&=\gamma\left(t\right) \left[f\left(\theta\left(v\left(t\right),t\right)\right) - f\left(\theta_{c}\right)\right] + \psi\left(v\left(t\right),t\right)  + V\left(t\right) + \int_{t_{0}}^{t}\dot{\gamma}\left(\tau\right)D_{f}\left(\theta_{c},\theta\left(v,\tau\right)\right)d\tau  -\int_{t_{0}}^{t} \frac{\partial}{\partial \tau}\psi\left(v,\tau\right)d\tau\\
			&\quad + \frac{1}{2}\int_{t_{0}}^{t}x^{T}\left(\tau\right)Qx\left(\tau\right)d\tau
		\end{aligned}
	\end{equation}
	Since $P>0$ so $V\left(t\right)\geq 0$, $Q>0$, $\psi\left(v\left(t\right),t\right)\geq \underline{\psi}$ and $\frac{\partial}{\partial t}\psi\left(v\left(t\right),t\right) \leq 0$ by definition, $\dot{\gamma}\left(t\right)>0$ by definition, and $D_{f}\left(\theta_{c},\theta\right)\geq 0$ as discussed in Eq. \eqref{Eq:Bregman_Ieq}, the conservation law of Eq. \eqref{Eq:consv} indicates that
	\begin{equation} \label{Eq:consv_Ieq}
		E \geq \gamma\left(t\right) \left[f\left(\theta\left(v\left(t\right),t\right)\right) - f\left(\theta_{c}\right)\right] + \underline{\psi}
	\end{equation}
	As $\gamma\left(t\right)>0$ for $\forall t \geq t_{0}$ by definition, Eq. \eqref{Eq:consv_Ieq} ensures the convergence rate bound for the objective function value as
	\begin{equation} \label{Eq:conv_rate}
			f\left(\theta\left(t\right)\right) - f\left(\theta_{c}\right) \leq \frac{E-\underline{\psi}}{\gamma\left(t\right)}
			=\frac{\gamma_{0}\left[f\left(\theta_{0}\right) - f\left(\theta_{c}\right)\right] + \psi_{0}-\underline{\psi}+\frac{1}{2}x_{0}^{T}Px_{0}}{\gamma\left(t\right)}
	\end{equation}
	where $\theta\left(t_{0}\right)=\theta_{0}$, $x\left(t_{0}\right)=x_{0}$, $\psi\left(v\left(t_{0}\right),t_{0}\right)=\psi_{0}$, and $\gamma\left(t_{0}\right)=\gamma_{0}$. One may also establish that 
	\begin{equation} \label{Eq:consv_Ieq_tilde_theta}
		E \geq \psi\left(v\left(t\right),t\right) = \psi\left(\gamma\left(t\right)\tilde{\theta}\left(t\right),t\right)
	\end{equation}
	where $\tilde{\theta}\left(t\right) \triangleq \theta\left(t\right) - \theta_{c}$. Note that Eq. \eqref{Eq:consv_Ieq_tilde_theta} can be utilised to certify the boundedness of $\tilde{\theta}\left(t\right)$ with appropriate choice of the function $\psi$. For example, let $\psi\left(v,t\right)=\psi\left(\gamma \tilde{\theta},t\right)=R\left(t\right)\left\|\tilde{\theta}\right\|_{p}^{q}$ where $R\left(t\right)>0$, $\dot{R}\left(t\right)\leq 0$, $p\geq 1$, and $q >0$. Then, considering the absolute homogeneity of a norm and the fact that $\gamma\left(t\right)>0$ by definition, Eq. \eqref{Eq:consv_Ieq_tilde_theta} ensures the convergence rate bound for the optimisation variable as
	\begin{equation} \label{Eq:conv_rate_tilde_theta}
		\left\|\tilde{\theta}\left(t\right)	\right\|_{p} \leq \frac{1}{\gamma\left(t\right)}\left(\frac{E}{R\left(t\right)}\right)^{\frac{1}{q}}
	\end{equation}
	
	It is impossible to obtain the actual optimisation variable $\theta$ directly from the relation in Eq. \eqref{Eq:theta_v}. For implementation of the process, we instead opt to the propagation of the following dynamics from a given initial guess $\theta_{0}$.
	\begin{equation} \label{Eq:theta_dot}
		\begin{aligned}
			\dot{\theta}\left(t\right) &= \frac{d}{dt}\theta\left(v\left(t\right),t\right) = \frac{d}{dt}\left(\gamma^{-1}\left(t\right)v\left(t\right)\right)\\
			&=-\gamma^{-2}\left(t\right)\dot{\gamma}\left(t\right)v\left(t\right) + \gamma^{-1}\left(t\right)\alpha\left(t\right)y\left(t\right)
		\end{aligned}
	\end{equation} 
	As a summary, Fig. \ref{Fig_BlkDiag} shows the block diagram representation of the passivity-based convex optimisation algorithm.
	\begin{figure}[hbt]
		\begin{center}
			\includegraphics[width=0.7\textwidth]{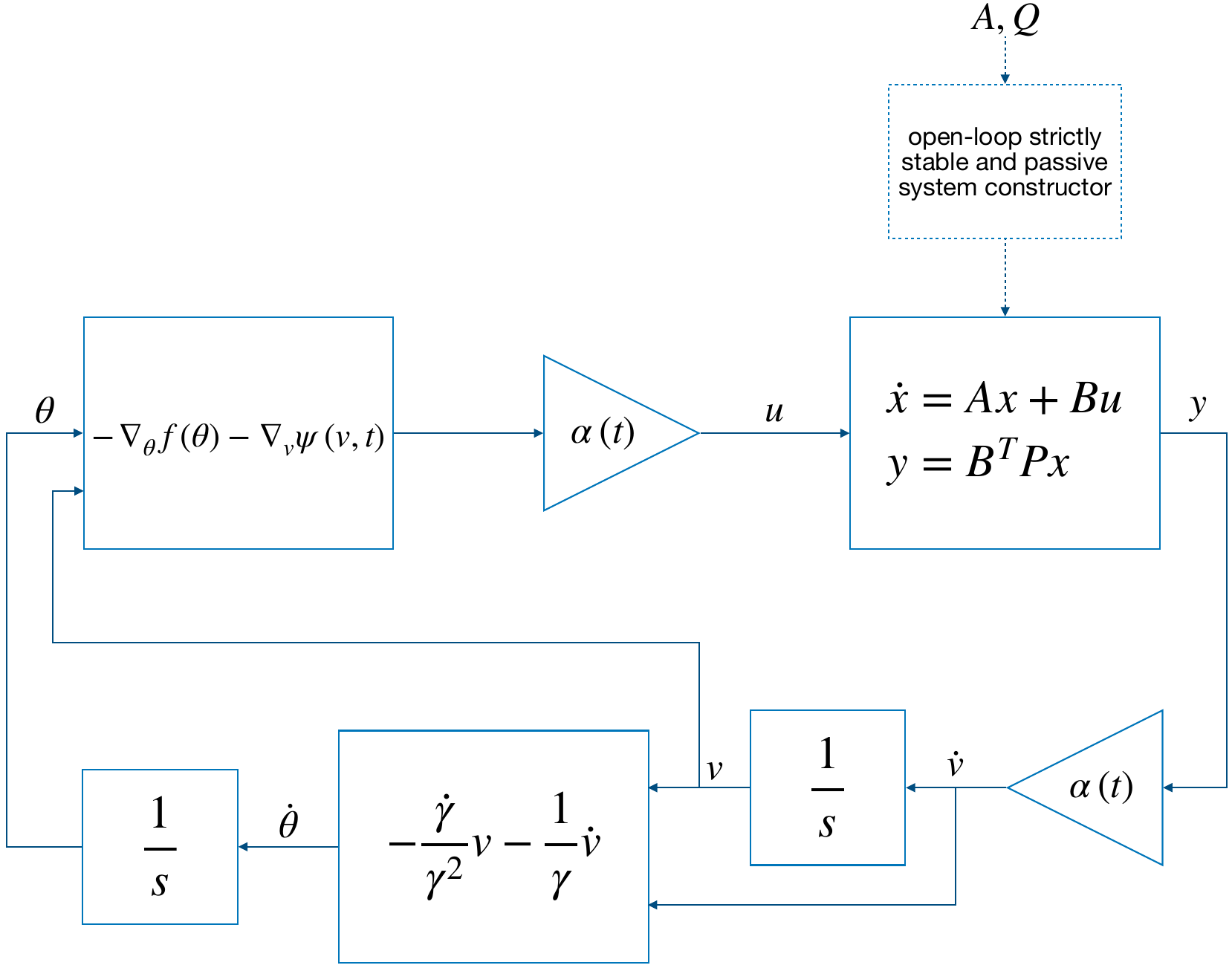}
			\caption{Block Diagram of Passivity-Based Convex Optimiser}
			\label{Fig_BlkDiag}
		\end{center}
	\end{figure}
	
	\subsection{Particular Case: $m=1$}
	Consider the particular case of $m=1$. In this case, the definitions of the variables reduce to $x=w$, $A=-a_{0}I$ for some $a_{0}>0$, $B=I$, $C=B^{T}P=P=\frac{Q}{2a_{0}}>0$, $y=Px$, and $u=-\alpha\left(\nabla_{\theta}f\left(\theta\right)+\nabla_{v}\psi\left(v,t\right)\right)$. 
	As a result, the optimisation dynamics can be written as
	\begin{align}
		\dot{x}\left(t\right) &= -a_{0}x\left(t\right) -\alpha\left(t\right)\left(\nabla_{\theta}f\left(\theta \right)+\nabla_{v}\psi\left(v,t\right)\right) \label{Eq:x_dot_m1}\\
		\dot{v}\left(t\right)&=\alpha\left(t\right)Px\left(t\right) \label{Eq:v_dot_m1}\\
		\dot{\theta}\left(t\right) &= -\gamma^{-2}\left(t\right)\dot{\gamma}\left(t\right)v\left(t\right) + \gamma^{-1}\left(t\right)\alpha\left(t\right)Px\left(t\right) \label{Eq:theta_dot_m1}
	\end{align}
	
	From Eq. \eqref{Eq:v_dot_m1}, we have
	\begin{equation} \label{Eq:x_v_dot}
		\begin{aligned}
			x &= \alpha^{-1}P^{-1}\dot{v}\\
			\dot{x} &= -\alpha^{-2}\dot{\alpha}P^{-1}\dot{v} + \alpha^{-1}P^{-1}\ddot{v}
		\end{aligned}
	\end{equation}
	Also, from Eq. \eqref{Eq:theta_v}, we have
	\begin{equation} \label{Eq:theta_v_dot}
		\begin{aligned}
			v&=\gamma \tilde{\theta}\\
			\dot{v} &= \dot{\gamma}\tilde{\theta} + \gamma\dot{\theta}\\
			\ddot{v} &= \ddot{\gamma}\tilde{\theta} + 2\dot{\gamma}\dot{\theta}+\gamma\ddot{\theta}
		\end{aligned}
	\end{equation}
	Substituting Eq. \eqref{Eq:x_v_dot} into Eq. \eqref{Eq:x_dot_m1} and rearranging the result leads to
	\begin{equation} \label{Eq:v_ddot}
			\ddot{v} +\left(a_{0}- \frac{\dot{\alpha}}{\alpha} \right)\dot{v}
			+ P\alpha^{2}\left(\nabla_{\theta}f\left(\theta \right)+\nabla_{v}\psi\left(v,t\right)\right) =0
	\end{equation}
	Again, substituting Eq. \eqref{Eq:theta_v_dot} into Eq. \eqref{Eq:v_ddot} and rearranging the result gives
	\begin{equation} \label{Eq:theta_ddot}
			\ddot{\theta} + \left(\frac{2\dot{\gamma}}{\gamma}+a_{0}-\frac{\dot{\alpha}}{\alpha}\right)\dot{\theta} + \left(\frac{\ddot{\gamma}}{\gamma}+\left(a_{0}-\frac{\dot{\alpha}}{\alpha}\right)\frac{\dot{\gamma}}{\gamma}\right)\tilde{\theta}
			 +\frac{P\alpha^{2}}{\gamma}\left(\nabla_{\theta}f\left(\theta \right)+\nabla_{v}\psi\left(v,t\right)\right) =0
	\end{equation}
	The equation above clearly shows that the optimisation dynamics for this particular case is a second-order dynamics in $\theta$ which can still capture many existing algorithms through the choice of $a_{0}$, $\gamma\left(t\right)$, $\alpha\left(t\right)$, and $\psi\left(v\left(t\right),t\right)$.
	
	\subsubsection{Example 1}
	If we take $\psi\left(v\left(t\right),t\right) = \frac{1}{2}\frac{R\left(t\right)}{\gamma^{2}\left(t\right)}\left\|v\left(t\right)\right\|_{2}^{2} + \delta$ for a function $R\left(t\right)$ such that $\dot{R}\left(t\right)\leq 0$ and some constant $\delta$, then Eq. \eqref{Eq:theta_ddot} turns into
	\begin{equation} \label{Eq:theta_ddot_psi_Rv}
		\begin{aligned}		
			\ddot{\theta} + \left(\frac{2\dot{\gamma}}{\gamma}+a_{0}-\frac{\dot{\alpha}}{\alpha}\right)\dot{\theta} 
			+\left(\frac{\ddot{\gamma}}{\gamma}+\left(a_{0}-\frac{\dot{\alpha}}{\alpha}\right)\frac{\dot{\gamma}}{\gamma}+\frac{P\alpha^{2}R}{\gamma^{2}}\right)\tilde{\theta} 
			+\frac{P\alpha^{2}}{\gamma}\nabla_{\theta}f\left(\theta \right) =0
		\end{aligned}
	\end{equation}
	
	\subsubsection{Example 2}
	If we choose $\alpha$ to be a constant and take $\psi\left(v,t\right)=0$, the time-derivatives of $\alpha$ disappear from Eq. \eqref{Eq:theta_ddot}, and we get
	\begin{equation} \label{Eq:theta_ddot_psi_Rv_alpha_const}
			\ddot{\theta} + \frac{2\dot{\gamma}+a_{0}\gamma}{\gamma}\dot{\theta} +\frac{\ddot{\gamma}+a_{0}\dot{\gamma}}{\gamma}\tilde{\theta}+\frac{P\alpha^{2}}{\gamma}\nabla_{\theta}f\left(\theta \right) =0
	\end{equation}
	
	\subsubsection{Example 3}
	Suppose that $\gamma\left(t\right)=\frac{c}{2}\alpha^{2}\left(t\right)$ for some constant $c>0$. In this setup, we have
	\begin{equation} \label{Eq:gamma_alpha_dot}
		\begin{aligned}
			\frac{\dot{\gamma}}{\gamma} &= \frac{2\dot{\alpha}}{\alpha}\\
			\frac{\ddot{\gamma}}{\gamma} &= 2\frac{\dot{\alpha}^{2}+\alpha\ddot{\alpha}}{\alpha^{2}}\\
		\end{aligned}
	\end{equation}
	Plugging Eq. \eqref{Eq:gamma_alpha_dot} into Eq. \eqref{Eq:theta_ddot} brings
	\begin{equation} \label{Eq:theta_ddot_alpha}
		\begin{aligned}		
			\ddot{\theta} + \left(a_{0}+\frac{3\dot{\alpha}}{\alpha}\right)\dot{\theta} + \frac{2\left(\ddot{\alpha}+a_{0}\dot{\alpha}\right)}{\alpha}\tilde{\theta} +\frac{2P}{c}\left(\nabla_{\theta}f\left(\theta \right)+\nabla_{v}\psi\left(v,t\right)\right) =0
		\end{aligned}
	\end{equation}
	Let $\psi\left(v,t\right)=0$ for simplicity and consider the limiting case of $a_{0} \rightarrow 0$. By choosing $\alpha\left(t\right) = kt$ with some nonzero constant $k$ and setting $c=2P$, Eq. \eqref{Eq:theta_ddot_alpha} reduces to
	\begin{equation} \label{Eq:theta_ddot_alpha_Nesterov}
		\ddot{\theta}\left(t\right) + \frac{3}{t}\dot{\theta}\left(t\right) + \nabla_{\theta}f\left(\theta\left(t\right)\right)=0
	\end{equation}
	which recovers the well-known result presented in \cite{Su2016,Suh2022} as the continuous-time limit of the Nesterov's accelerated gradient descent method. Since $\gamma\left(t\right) = Pk^{2}t^{2}$ in this case, Eq. \eqref{Eq:conv_rate} with $\theta_{c} = \theta_{*}$ certifies the convergence rate as
	\begin{equation} \label{Eq:conv_rate_Nesterov}
		f\left(\theta\left(t\right)\right) - f_{*} \leq \frac{E}{Pk^{2}t^{2}} \sim \mathcal{O}\left(\frac{1}{t^{2}}\right)
	\end{equation}
	
\begin{rem}
Design parameters of the accelerated gradient flow dynamics in Eq. \eqref{Eq:theta_ddot} should be chosen carefully to ensure realisability and robustness against noisy gradients. First, the coefficient of the third term in Eq. \eqref{Eq:theta_ddot} should be zero to make the algorithm implementable since $\tilde{\theta}$ cannot be measured. Second, establishing uniform asymptotic stability is advantageous to ensure robustness of the time-varying optimisation dynamics against disturbances \cite{Poveda2019,Poveda2020}.
\end{rem}
	
	\subsection{Particular Case: $m=0$}
	The particular case of $m=0$ corresponds to the case where the generator $\mathcal{G}$ is a static function, i.e., memoryless map, between the input-output pairs. In this special case, the system $y=\mathcal{G}\left[u\right]$ is defined to be passive if $u^{T}y \geq 0$, and the notion of storage function is no longer necessary. By understanding that Eq. \eqref{Eq:w_sys} becomes an algebraic relation of the form $w=u$, a passive map can be constructed naturally by defining the output as $y=w=u$ since $u^{T}y=u^{T}u\geq 0$. Without changing the definition of the signals and functions introduced to be compatible with the passivity of the input-output pair, the main conservation result given by Eqs. \eqref{Eq:consv} reduces to
	\begin{equation} \label{Eq:consv_m0}
		\begin{aligned}
			E &\equiv U\left(v\left(t_{0}\right),t_{0}\right)\\
			&=\gamma\left(t\right) \left[f\left(\theta\left(v\left(t\right),t\right)\right) - f\left(\theta_{c}\right)\right] + \psi\left(v\left(t\right),t\right)   + \int_{t_{0}}^{t}\dot{\gamma}\left(\tau\right)D_{f}\left(\theta_{c},\theta\left(v,\tau\right)\right)d\tau -\int_{t_{0}}^{t} \frac{\partial}{\partial \tau}\psi\left(v,\tau\right)d\tau  
		\end{aligned}
	\end{equation}
	which verifies
	\begin{align}
		E &\geq \gamma\left(t\right) \left[f\left(\theta\left(v\left(t\right),t\right)\right) - f\left(\theta_{c}\right)\right] + \underline{\psi}\\
		E &\geq \psi\left(v\left(t\right),t\right)
	\end{align}
	Also, the optimisation dynamics in $\theta$ can be expressed by combining Eqs. \eqref{Eq:v_dot}, \eqref{Eq:grad_U}, \eqref{Eq:theta_v_dot}, and the relation $y=u$ as
	\begin{equation} \label{Eq:theta_dot_m0}
		\dot{\theta} + \frac{\dot{\gamma}}{\gamma}\tilde{\theta} + \frac{\alpha^{2}}{\gamma}\left(\nabla_{\theta}f\left(\theta\right)+\nabla_{v}\psi\left(v,t\right)\right) = 0
	\end{equation}
	which recovers the simple gradient flow dynamics given by $\dot{\theta}+\nabla_{\theta}f\left(\theta\right)=0$ if $\gamma=\alpha^{2}$ is constant and $\psi\left(v,t\right)=0$.
	
	\section{Extension to Adaptive Control} \label{Sec:AC}
	Adaptive control deals with stable tracking of an uncertain dynamic system by means of online parameter update. Hence, the corresponding closed-loop dynamics generally consists of tracking error and parameter estimation error subsystems. If the uncertain part of the plant dynamics can be represented with a linear parametric model, parameter update for stable adaptive control can be interpreted as a convex optimisation process which involves interaction with an external system.
	
	Consider the common tracking error model of an adaptive control system given by
	\begin{equation} \label{Eq:e_dot}
		\dot{e}\left(t\right) = A_{m} e\left(t\right) + B_{p}\Phi\left(x_{p}\right)\tilde{\theta}\left(t\right)
	\end{equation} 
	where $A_{m}$ is a constant Hurwitz system matrix of the closed-loop plant dynamics, $B_{p}$ is the constant control effectiveness matrix of the plant, $\Phi$ represents the known basis function for the linear model of the plant uncertainty, and $x_{p}\left(t\right)$ denotes the plant state. It is assumed that the true uncertainty is represented with a constant true parameter $\theta_{*}=\theta_{c}$. Lemma \ref{Lem:Lyap} states that there exists $P_{e}=P_{e}^{T}>0$ verifying $A_{m}^{T}P_{e}+P_{e}A_{m}=-Q_{e}$ for any given $Q_{e}=Q_{e}^{T} > 0$ and the corresponding Lyapunov function is $V_{e} \triangleq \frac{1}{2}e^{T}P_{e}e$. The time-derivative of the Lyapunov function can be written as
	\begin{equation} \label{Eq:V_e_dot}
		\begin{aligned}
			\dot{V}_{e} &= -\frac{1}{2}e^{T}Q_{e}e + e^{T}P_{e}B_{p}\Phi\tilde{\theta}\\
			&= -\frac{1}{2}e^{T}Q_{e}e + \left<\tilde{\theta},\nabla_{\theta}\dot{V}_{e}\right>
		\end{aligned}
	\end{equation}
	One may notice that $\dot{V}_{e}$ alone is sign-indefinite.
	
	To design a parameter update algorithm based on accelerated gradient flow, it is necessary to allow the optimisation dynamics to have a port for interaction with exogenous signal. With this background, let us redefine the output of the generator $\mathcal{G}$ inside the optimiser by introducing a new variable $z$ to replace Eq. \eqref{Eq:v_dot} with 
	\begin{equation} \label{Eq:v_dot_z}
		\dot{v}\left(t\right)+z\left(t\right) = \alpha\left(t\right) y\left(t\right)
	\end{equation}
	while keeping the input definition unchanged as shown in Eq. \eqref{Eq:xu}. With this setup, Eq. \eqref{Eq:W_dot} can be rewritten as 
	\begin{equation} \label{Eq:W_dot_z}
		\begin{aligned}
			\dot{W}	&= \dot{V} - \left<y,u\right> = \dot{V} + \left<\dot{v}+z,\nabla_{v}U\left(v,t\right)\right>\\
			&= \dot{V} + \frac{d}{dt}U\left(v,t\right)-\frac{\partial}{\partial t}U\left(v,t\right)  + \left<z,\nabla_{v}U\left(v,t\right)\right>\\
			&= -\frac{1}{2}x^{T}Qx \\
			&\leq 0
		\end{aligned}
	\end{equation}
	Since we already have Eq. \eqref{Eq:part_U_t} that states $\frac{\partial}{\partial t} U\left(v,t\right) \leq 0$, the following inequality also holds
	\begin{equation} \label{Eq:W_dot_z_part_U}
		\begin{aligned}
			\dot{W} + \frac{\partial}{\partial t} U\left(v,t\right)  &= \dot{V} + \frac{d}{dt}U\left(v,t\right) + \left<z,\nabla_{v}U\left(v,t\right)\right>\\
			&=-\frac{1}{2}x^{T}Qx  -\dot{\gamma} D_{f}\left(\theta_{c},\theta \right) + \frac{\partial}{\partial t}\psi\left(v ,t\right) \\
			&\leq 0
		\end{aligned}
	\end{equation}
	In the followings, let $\psi\left(v,t\right)=0$ for simplicity, the choice which verifies $\nabla_{v}U\left(v,t\right) = \nabla_{\theta}f\left(\theta\right)$ according to Eq. \eqref{Eq:grad_U}. Then, by rearranging Eq. \eqref{Eq:W_dot_z_part_U}, we have
	\begin{equation} \label{Eq:V_U_dot}
		\begin{aligned}
			\dot{V} + \frac{d}{dt}U\left(v,t\right) &=-\frac{1}{2}x^{T}Qx  -\dot{\gamma} D_{f}\left(\theta_{c},\theta \right)-\left<z,\nabla_{\theta}f\left(\theta\right)\right>\\
			&\leq -\left<z,\nabla_{\theta}f\left(\theta\right)\right>
		\end{aligned}
	\end{equation}
	
	Now, let us define the Lyapunov function for the total closed-loop system by
	\begin{equation} \label{Eq:V_T_defn}
		V_{T} \triangleq V + U + V_{e}
	\end{equation}
	Also, let us define the additional variable as $z = \tilde{\theta} = \frac{v}{\gamma}$ and choose the objective function as $f\left(\theta\right) = \dot{V}_{e} + L\left(\theta\right)$ where $L$ is a loss function representing the goal of uncertainty learning which should be measurable. One example for the learning loss based on regressor extension is $L\left(\theta\right)=\frac{1}{2}\tilde{Y}^{T}\tilde{Y}=\frac{1}{2}\tilde{\theta}^{T}\Omega^{2}\tilde{\theta}$ where $\Omega=\mathcal{F}\left[\Phi^{T}\Phi\right]$ with some stable linear filter $\mathcal{F}$. Then, by using Eqs. \eqref{Eq:V_e_dot} and \eqref{Eq:V_U_dot}, the time-derivative of Eq. \eqref{Eq:V_T_defn} satisfies
	\begin{equation} \label{Eq:V_T_dot}
		\begin{aligned}
			\dot{V}_{T} &= -\frac{1}{2}x^{T}Qx   -\frac{1}{2}e^{T}Q_{e}e  -\dot{\gamma} D_{f}\left(\theta_{c},\theta \right) -\tilde{\theta}^{T}\Omega^{2}\tilde{\theta}\\
			&\leq 0
		\end{aligned}
	\end{equation}
	for $\forall t \geq t_{0}$. Integrating Eq. \eqref{Eq:V_T_dot} yields
	\begin{equation} \label{Eq:V_T_Ieq}
		\begin{aligned}
			E_{AC} &\equiv  V\left(t_{0}\right) + U\left(v\left(t_{0}\right),t_{0}\right) + V_{e}\left(t_{0}\right)\\
			&\geq V\left(t\right) + U\left(v\left(t\right),t\right) + V_{e}\left(t\right) \\
			&\geq U\left(v\left(t\right),t\right) \\
			&= \gamma\left(t\right)\left[f\left(\theta\left(t\right)\right) - f\left(\theta_{c}\right)\right] 
		\end{aligned}
	\end{equation}
	As $\gamma\left(t\right)>0$ for $\forall t \geq t_{0}$ by definition, Eq. \eqref{Eq:V_T_Ieq} ensures the convergence rate bound for the objective function value as
	\begin{equation} \label{Eq:conv_rate_AC}
		f\left(\theta\left(t\right)\right) - f\left(\theta_{c}\right) \leq \frac{E_{AC}}{\gamma\left(t\right)}
	\end{equation}
	
	If the plant controller is driven by an exogenous reference command under which the system experiences a sufficient but not necessarily persistent excitation, then we can introduce a mild assumption that $\Omega > 0$ will be satisfied after some finite time $t_{e}$. Therefore, assuming the presence of interval excitation, Eq. \eqref{Eq:V_T_dot} indicates the negative definiteness of $\dot{V}_{T}$ for $\forall t\geq t_{e}$. Consequentially, the convergence of $\left(x,e,\tilde{\theta}\right)$ towards $\left(0,0,0\right)$ can be established.

	\section{Conclusion} \label{Eq:Concls}
	This study developed a systematic method for constructing an accelerated algorithm in continuous-time domain to solve convex optimisation problems. The proposed method takes a control-theoretic approach which is to synthesise a passive dynamics as an internal generator with compatible choice of the input-output pair. The classical notion of passivity along with the Kalman-Yakubovich-Popov lemma plays a key role not only in the definition of output of the generator dynamics but also in the characterisation of the storage function. The convergence rate bound for the function value was established in relation to the time-varying factor that relates the integral of scaled output with the optimisation variable. As an example, the ordinary differential equation model for the Nesterov's accelerated gradient descent method is shown as a particular case. The passivity-based optimisation methodology has its significance as i) a theoretical bridge between optimisation and control and also as ii) a generic framework that can be utilised to design convex optimisation algorithms.

	\bibliographystyle{new-aiaa}
	\bibliography{KYP_Opt}

\end{document}